\documentclass[a4paper]{amsart}

\usepackage{amsmath,amssymb}
\usepackage{amsthm}

\usepackage{ascmac}
\usepackage[margin=30mm]{geometry}

\usepackage{amsmath}
\usepackage{enumerate}
\usepackage{comment}
\usepackage[all]{xy}

\begin{document}

\theoremstyle{definition}
\newtheorem{theorem}{Theorem}[section]
\newtheorem*{theorem*}{Theorem}
\newtheorem{definition}[theorem]{Definition}
\newtheorem*{definition*}{Definition}
\newtheorem{proposition}[theorem]{Proposition}
\newtheorem*{proposition*}{Proposition}
\newtheorem{lemma}[theorem]{Lemma}
\newtheorem*{lemma*}{Lemma}
\newtheorem{claim}[theorem]{Claim}
\newtheorem*{claim*}{Claim}
\newtheorem{corollary}[theorem]{Corollary}
\newtheorem*{corollary*}{Corollary}
\newtheorem{remark}[theorem]{Remark}
\newtheorem*{remark*}{Remark}
\newtheorem{example}[theorem]{Example}
\newtheorem*{example*}{Example}
\newtheorem{refproof}{Proof}

  \title{The $G$-index, $G$-coindex and $G$-weight of a real moment-angle complex}
  \author{Akatsuki Kizu}
  \address{Department of Mathematics, Kyoto University, Kyoto, 606-8502, Japan}
  \email{kizu.akatsuki.85e@st.kyoto-u.ac.jp}
  \date{}
  \maketitle

\begin{abstract}
A real moment-angle complex is a CW complex defined by the combinatoric information of a simplicial complex, and has a standard action of $2$-torus. For each subtorus $G$ of $2$-torus, we describe invariants corcerning this action, the $G$-index, $G$-coindex and $G$-weight, combinatorially. 
\end{abstract}


\section{Introduction and the main theorems}

\quad Let $K$ be a simplicial complex on the vertex set $[m]=\{1, \dots, m\}$. A subset $I\subset[m]$ is called a \textit{face} of $K$ if $I \in K$. Otherwise, it is called a \textit{non-face} of K. We say that a subset $I\subset[m]$ is a \textit{minimal non-face} of $K$ if it is a non-face of $K$ and every proper subset of $I$ is a face of $K$. The \textit{dimension} of $I$ is defined as 1 less than its order. If one point subset $\{i\} \subset [m]$ is a non-face of $K$, it is called a \textit{ghost vertex} of $K$. In this paper, we assume that every simplicial complex does not have ghost vertices. 

A \textit{real moment-angle complex} associated with $K$  is defined as a subcomplex of $(D^1)^m$
$$
\mathbb{R}Z_K:=\bigcup_{\sigma \in K}(D^1, S^0)^{\sigma}, 
$$ 
where $(D^1, S^0)^{\sigma}:=\{(x_1, \dots, x_m) \in (D^1)^m \mid \text{$x_i \in S^0$ if $i \notin \sigma$}\}$. As seen in the paper of Davis and Januszukiewicz \cite{DJ}, $\mathbb{R}Z_K$ is one of the key objects and has been deeply studied in toric topology. Since $\mathbb{R}Z_K$ is constructed by using the combinatoric information of $K$, it is natural to expect that homotopy invariants of $\mathbb{R}Z_K$ could be described in terms of the combinatorics of $K$. For example, the equivariant cohomology is studied in the paper of Davis and Januszukiewicz \cite{DJ}. The cohomology ring is studied in \cite{cai} in the case that $\mathbb{R}Z_K$ is a topological manifold. The homotopy type is studied in \cite{grbic} when $K$ is \textit{flag}, that is, every minimal non-face of $K$ consists of two vertices. In \cite{iriye}, it is shown that $\mathbb{R}Z_K$ plays an important role to understand the homotopy type of a  \textit{polyhedral product}, a generalized object of a real moment-angle complex,  by using a structure which is called \textit{fat-wedge filtration}. 

By construction, there is a standard action of $2$-torus $(\mathbb{Z}/2)^m$ on $\mathbb{R}Z_K$. Specifically, for $g=(g_1, \dots, g_m) \in (\mathbb{Z}/2)^m$ and $x=(x_1, \dots, x_m) \in \mathbb{R}Z_K$, we set 
$$
g \cdot x:=\left( g_1x_1, \dots, g_mx_m \right), 
$$
where we regard $\mathbb{Z}/2=\{1, -1\}$ and $D^1=[-1, 1]$. Since $K$ has no ghost vertex, it can be easily seen that this action is always not free. However, some subtori of $(\mathbb{Z}/2)^m$ may act on $\mathbb{R}Z_K$ freely. The quotient space of $\mathbb{R}Z_K$ by the free action of a subtorus of $(\mathbb{Z}/2)^m$, called \textit{partial quotient} in \cite{buchstaber}, has been studied in the many context \cite{franz, fu, hasui1}. In particular, small cover which is introduced in \cite{DJ} is a classical example of partial quotient and has been deeply studied \cite{choi, lu}. 

Let $G$ be a non-trivial discrete group and $X$ be a free $G$-space. We define the topological invariant \textit{$G$-index} as
$$
\mathrm{ind}_{G}(X):=\min \{n \geq 0 \mid \text{there is a $G$-map $X \to E_nG$} \}, 
$$
where $E_nG$ is the $(n+1)$-fold join of $G$ with a diagonal $G$-action. We also consider the similar notion \textit{$G$-coindex}. Let $X$ be a $G$-space. Define 
$$
\mathrm{coind}_{G}(X):=\min \{n \geq 0 \mid \text{there is a $G$-map $E_nG \to X$} \}. 
$$
These invariants were introduced in the paper of Conner and Floyd \cite{conner}. Note that they call $G$-index what we call $G$-coindex in \cite{conner}, and vice versa. As will be seen later, there is an inequality $\mathrm{coind}_{G}(X) \leq \mathrm{ind}_{G}(X)$, which is a generalization of the Borsuk-Ulam theorem. The $G$-index has a close relationship with Lusternik-Shnirelmann category; see \cite[Section 5.3]{matousek}. The $G$-index and $G$-coindex are so hard to compute in general. We also introduce the cohomological invariant \textit{$G$-weight}. Recall that there is a fibration $EG \times_{G} X \to BG$ for a $G$-space $X$. Define
$$
\mathrm{wgt}_{G}(X):=\max \{n \geq 0 \mid \text{$H^k(BG) \to H^k(EG \times_{G} X)$ is injective for all $ k\leq n$ }\}. 
$$
This invariant is essentially the same as the ideal-valued index in the paper of Fadell and Husseini \cite{fadell}. As will be seen later, there are inequalities $\mathrm{coind}_{G}(X) \leq \mathrm{wgt}_{G}(X) \leq \mathrm{ind}_{G}(X)$ when $G$ is a $p$-torus. Thus the $G$-weight is helpful to evaluate the $G$-index and $G$-coindex. However, the $G$-weight is still hard to compute, not as hard as the $G$-index and $G$-coindex. We can see that the $G$-weight is used in \cite{hasui2, volovikov1, volovikov2}. We will see some properties of these invariants in Section 2.  

Our goal in this paper is to describe the $G$-index, $G$-coindex and $G$-weight of $\mathbb{R}Z_K$ in terms of the combinatorics of $K$ for a given subgroup $G < (\mathbb{Z}/2)^m$. Let us introduce the combinatorial notion \textit{$\delta$-number $\delta(K)$} of a simplicial complex $K$ defined as the minimum number of the dimension of non-faces of $K$, that is, 
$$ 
\delta(K):= \min \left\{|I|-1\mid \text{$I$ is a non-face of $K$} \right\}. 
$$
If $K$ is a full simplex, then we set $\delta(K)=\infty$. 

We now can state the main theorems. For a non-empty subset $I \subset [m]$, we define the \textit{full subcomplex} of $K$ over $I$ as 
$$
K_I :=\{\sigma \subset I \mid \sigma \in K\}. 
$$
For each $g=(g_1, \dots, g_m) \in G$, let us set the \textit{support} of $g$ as the subset of $[m]$ 
$$
\mathrm{supp}(g):=\{i \in [m] ~|~ g_i=-1 \}, 
$$
and set $\mathrm{supp}(G):= \bigcup_{g \in G}\mathrm{supp}(g)$.  

\begin{theorem} \label{main1}
\textit{
Let $G < (\mathbb{Z}/2)^m$ be a subgroup of rank one. 
\begin{enumerate}[(i)]
	\item If the action of $G$ on $\mathbb{R}Z_K$ is free, then there is an equality
$$
\mathrm{ind}_{G}(\mathbb{R}Z_K)=\delta(K_{\mathrm{supp}(G)}). 
$$
	\item There are equalities
$$
\mathrm{coind}_{G}(\mathbb{R}Z_K)=\mathrm{wgt}_{G}(\mathbb{R}Z_K)=\delta(K_{\mathrm{supp}(G)}), 
$$
where the action of $G$ on $\mathbb{R}Z_K$ is not necessarily free. 
\end{enumerate}
}
\end{theorem}

\begin{theorem} \label{main2}
\textit{
	Let $G < (\mathbb{Z}/2)^m$ be a non-trivial subgroup. 
\begin{enumerate}[(i)]
	\item If the action of $G$ on $\mathbb{R}Z_K$ is free, then there are inequalities
$$
\max_{g \in G \setminus \{1\}} \delta(K_{\mathrm{supp}(g)}) \leq \mathrm{ind}_{G}(\mathbb{R}Z_K) \leq \dim K+1. 
$$
	\item There are inequalities
$$ 
\delta(K_{\mathrm{supp}(G)}) \leq \mathrm{coind}_{G}(\mathbb{R}Z_K) \leq \mathrm{wgt}_{G}(\mathbb{R}Z_K) \leq \min_{g \in G \setminus \{1\}} \delta(K_{\mathrm{supp}(g)}), 
$$
where the action of $G$ on $\mathbb{R}Z_K$ is not necessarily free. 
\end{enumerate}
}
\end{theorem}
These theorems say that the $G$-index, $G$-coindex and $G$-weight for a real moment-angle complex can be evaluated by using a $\delta$-number and especially be computed completely when the rank of $G$ is one. We will give proofs of the main theorems in Section 3.

In Section 4, we study the upper bound of the $G$-index. As shown in Section 2, if $X$ is a free $G$-complex, then its $G$-index is smaller than or equal to its dimension, that is, $\mathrm{ind}_G(X) \leq \dim X$. Thus if $G$ acts on $\mathbb{R}Z_K$ freely, then there is an inequality 
$$
\mathrm{ind}_G(\mathbb{R}Z_K) \leq \dim \mathbb{R}Z_K = \dim K+1. 
$$
We find some sufficient condition for making the upper bound smaller. A pair of simplices $(\sigma, \tau)$ of $K$ is said to be an \textit{elementary collapsing pair} if $\tau \subset \sigma$ and $\sigma$ is the only simplex including $\tau$ properly. When $(\sigma, \tau)$ is an elementary collapsing pair of $K$, removing $\sigma$ and $\tau$ from $K$ is called an \textit{elementary collapse}, denoted by $K \searrow_{(\sigma, \tau)} K^{\prime}$, where $K^{\prime}:=K \setminus \{\sigma, \tau\}$. A sequence of elementary collapses is simply said to be a \textit{collapse}. 

\begin{proposition}
\label{retract}
\textit{
Let $K$ be a simplicial complex and suppose that $G < (\mathbb{Z}/2)^m$ acts on $\mathbb{R}Z_K$ freely. If $K$ collapses to some simplicial complex $L$ such that $\dim L < \dim K$, then there is an inequality
$$
\mathrm{ind}_G(\mathbb{R}Z_K) \leq \dim K. 
$$    
}
\end{proposition}

Proposition \ref{retract} means that we can make the upper bound smaller by one under the condition above. We will see that this modification of the upper bound is best possible by giving some example. 

In Section 5, we also study the condition that the $G$-coindex and $G$-weight of a real moment-angle complex equals to the lower bound. 

\begin{corollary} 
\label{cor1}
\textit{
	Let $G < (\mathbb{Z}/2)^m$ be a non-trivial subgroup. If $\delta(K_{\mathrm{supp}(G)})=1$, then there are equalities
$$
\mathrm{coind}_{G}(\mathbb{R}Z_K)=\mathrm{wgt}_{G}(\mathbb{R}Z_K)=1.
$$ 
}
\end{corollary}

\begin{corollary} 
\label{cor2}
\textit{
	Let $G < (\mathbb{Z}/2)^m$ be a non-trivial subgroup. If each minimal non-face of $K_{\mathrm{supp}(G)}$ has the same order and $\mathrm{supp}(g_0) \notin K$ for some $g_0 \in G$, then there are equalities 
$$
\mathrm{coind}_{G}(\mathbb{R}Z_K)=\mathrm{wgt}_{G}(\mathbb{R}Z_K)=\delta(K_{\mathrm{supp}(G)}).  
$$ 
}
\end{corollary}


\section{Properties of the $G$-index, $G$-coindex and $G$-weight}

In this section, we show some properties of the $G$-index, $G$-coindex and $G$-weight. Throughout this section, let $G$ be a $p$-torus $(\mathbb{Z}/p)^m$, where $p$ is a prime number. We first observe a cohomological property of  the canonical map $i_n \colon B_nG \to BG$. 

\begin{lemma}
\label{B_nG}
The map $i_n^* \colon H^*(BG; \mathbb{Z}/p) \to H^*(B_nG; \mathbb{Z}/p)$ is injective for $* \leq n$ and not injective for $*=n+1$. 
\end{lemma}

\begin{proof}
There is a homotopy fibration $E_nG \to B_nG \xrightarrow{i_n} BG$. Since $E_nG$ is homotopy equivalent to a wedge of $n$-spheres, it is $(n-1)$-connected. Then the Serre exact sequence for the above homotopy fibration implies the first statement. On the other hand, it is well known that 
$$
H^*(B\mathbb{Z}/p; \mathbb{Z}/p) = \left\{
\begin{array}{ll}
\mathbb{Z}/p[x] & (p=2)\\[3pt]
\Lambda_{\mathbb{Z}/p}(x) \otimes \mathbb{Z}/p[y] & (p \neq 2), 
\end{array}
\right. 
$$
where $\deg x=1$ and $\deg y=2$, implying that $H^{n+1}(BG; \mathbb{Z}/p) \neq 0$. However, since $B_nG$ is an $n$-dimensional complex, $H^{n+1}(B_nG; \mathbb{Z}/p)=0$. Thus the second statement is proved. 
\end{proof}

Let a map $p_1 \colon EG \times_G E_nG \to BG $ be the first projection and a map $p_2 \colon EG \times_G E_nG \to B_nG$ be the second projection. 

\begin{lemma}
\label{commutative}
\textit{
The map $p_2 \colon EG \times_G E_nG \to B_nG$ is a homotopy equivalence and a diagram 
\begin{equation*}
\xymatrix@C=48pt{
	EG \times_G E_nG \ar[r]^-{p_2}\ar[d]_-{p_1} & B_nG \ar[d]^-{i_n}\\
	BG \ar@{=}[r] & BG
}
\end{equation*}
is homotopy commutative.
}
\end{lemma}

\begin{proof}
The map $p_2 \colon EG \times_G E_nG \to B_nG$ is a homotopy equivalence since its fiber $EG$ is contractible. The homotopy commutative diagram in the statement is obtained by taking the quotient of a diagram of $G$-maps
\begin{equation*}
\xymatrix@C=48pt{
	EG \times E_nG \ar[r]\ar[d] & E_nG \ar[d]\\
	EG \ar@{=}[r] & EG
}
\end{equation*}
which commutes up to $G$-homotopy by the fact that any two $G$-maps from some $G$-space to $EG$ are $G$-homotopic as in \cite[Lemma 5.1]{gottlieb}. Thus the proof is completed. 
\end{proof}

Let us show fundamental properties of the $G$-index, $G$-coindex and $G$-weight. 

\begin{proposition}
\label{property1}
\textit{
Let $X$ be a $G$-space. 
\begin{enumerate}[(i)]
  \setlength{\itemsep}{-0.75mm}
	\item There is an inequality 
$$
\mathrm{coind}_{G}(X) \leq \mathrm{wgt}_{G}(X).
$$ 
	\item If $G$ acts on $X$ freely, then there is an inequality 
$$
\mathrm{wgt}_{G}(X) \leq \mathrm{ind}_{G}(X). 
$$ 
\end{enumerate}
}
\end{proposition}

\begin{proof}
(i) Let $n:=\mathrm{coind}_{G}(X)$. Then there is a $G$-map $\phi \colon E_nG \to X$. By Lemma \ref{commutative}, we get a homotopy commutative diagram 
\begin{equation*}
\xymatrix@C=48pt{
	B_nG \ar[d]_-{i_n} & EG \times_G E_nG \ar[l]_-{p_2}^-{\simeq}\ar[r]^-{\mathrm{id} \times_G \phi}\ar[d]^-{p_1} & EG \times_G X \ar[d]\\
	BG \ar@{=}[r] & BG \ar@{=}[r] & BG
}
\end{equation*}
which induces a commutative diagram
\begin{equation*}
\xymatrix@C=54pt{
	H^k(B_nG; \mathbb{Z}/p) & H^k(EG \times_G X; \mathbb{Z}/p) \ar[l]\\
	H^k(BG; \mathbb{Z}/p) \ar[u]^-{i_n^*} \ar@{=}[r] & H^k(BG; \mathbb{Z}/p). \ar[u]
}
\end{equation*}
By the first statement of Lemma \ref{B_nG}, the map $H^k(BG; \mathbb{Z}/p) \to H^k(EG \times_G X; \mathbb{Z}/p)$ must be injective for $k \leq n$. Thus we obtain the inequality $\mathrm{wgt}_{G}(X) \geq n=\mathrm{coind}_{G}(X)$. 

(ii) Let $m:=\mathrm{ind}_{G}(X)$. Then there is a $G$-map $\psi \colon X \to E_mG$. By Lemma \ref{commutative}, we get a homotopy commutative diagram
\begin{equation*}
\xymatrix@C=48pt{
	EG \times_G X \ar[d]\ar[r]^-{\mathrm{id} \times_G \psi} & EG \times_G E_mG \ar[r]^-{p_2}_-{\simeq}\ar[d]_-{p_1} & B_mG\ar[d]^-{i_n} \\
	BG \ar@{=}[r] & BG \ar@{=}[r] & BG 
}
\end{equation*}
which induces a commutative diagram
\begin{equation*}
\xymatrix@C=54pt{
	H^k(B_mG; \mathbb{Z}/p)\ar[r] & H^k(EG \times_G X; \mathbb{Z}/p)\\
	H^k(BG; \mathbb{Z}/p) \ar[u]^-{i_n^*} \ar@{=}[r] & H^k(BG; \mathbb{Z}/p). \ar[u] 
}
\end{equation*}
By the second statement of Lemma \ref{B_nG}, the map $H^k(BG; \mathbb{Z}/p) \to H^k(EG \times_G X; \mathbb{Z}/p)$ cannot be injective for $k=m+1$. Thus we also obtain the inequality $\mathrm{wgt}_{G}(X) \leq m=\mathrm{ind}_{G}(X)$. 
\end{proof}

\begin{proposition}
\label{map}
\textit{
Let $X$ and $Y$ be $G$-spaces, and suppose that there is a $G$-map $f \colon X \to Y$. 
\begin{enumerate}[(i)]
  \setlength{\itemsep}{-0.75mm}
	\item There are inequalities 
$$
\mathrm{coind}_{G}(X) \leq \mathrm{coind}_{G}(Y) \quad \text{and} \quad \mathrm{wgt}_{G}(X) \leq \mathrm{wgt}_{G}(Y). 
$$ 
	\item If $G$ acts on $X$ and $Y$ freely, then there is an inequality 
$$
\mathrm{ind}_{G}(X) \leq \mathrm{ind}_{G}(Y). 
$$ 
\end{enumerate}
}
\end{proposition}

\begin{proof}
(i)~Let $n:=\mathrm{coind}_{G}(X)$, and take a $G$-map $\phi \colon E_nG \to X$. Then the composition $f \circ \phi \colon E_nG \to X \to Y$ is also a $G$-map. Thus we get the inequality $\mathrm{coind}_{G}(Y) \geq n=\mathrm{coind}_{G}(X)$. Let $l:=\mathrm{wgt}_{G}(X)$. Since the map $f \colon X \to Y$ is a $G$-map, it induces a commutative diagram 
\begin{equation*}
\xymatrix@C=54pt{
	EG \times_G X \ar[r]^-{\mathrm{id} \times_G f}\ar[d] & EG \times_G Y \ar[d]\\
	BG \ar@{=}[r] & BG. 
}
\end{equation*}
Then we get a commutative diagram
\begin{equation*}
\xymatrix@C=54pt{
	H^k(EG \times_G Y; \mathbb{Z}/p)\ar[r]^-{(\mathrm{id} \times_G f)^*} & H^k(EG \times_G X; \mathbb{Z}/p)\\
	H^k(BG; \mathbb{Z}/p) \ar[u] \ar@{=}[r] & H^k(BG; \mathbb{Z}/p). \ar[u] 
}
\end{equation*}
Since the map $H^k(BG; \mathbb{Z}/p) \to H^k(EG \times_G X; \mathbb{Z}/p)$ is injective for $k \leq l$, the map $H^k(BG; \mathbb{Z}/p) \to H^k(EG \times_G Y; \mathbb{Z}/p)$ must also be injective for $k \leq l$. Thus we get the inequality $\mathrm{wgt}_{G}(Y) \geq l=\mathrm{wgt}_{G}(X)$.

(ii)~Let $m:=\mathrm{ind}_{G}(Y)$, and take a $G$-map $\psi \colon Y \to E_mG$. Then the composition $\psi \circ f \colon X \to Y \to E_mG$ is also a $G$-map. Thus we get the inequality $\mathrm{ind}_{G}(X) \leq m=\mathrm{ind}_{G}(Y)$. 
\end{proof}

\begin{proposition}
\label{subgroup}
\textit{
Let $X$ be a $G$-space, and $K < G$ be a non-trivial subgroup. 
\begin{enumerate}[(i)]
  \setlength{\itemsep}{-0.75mm}
	\item There are inequalities 
$$
\mathrm{coind}_{G}(X) \leq \mathrm{coind}_{K}(X) \quad \text{and} \quad \mathrm{wgt}_{G}(X) \leq \mathrm{wgt}_{K}(X).
$$ 
	\item If $G$ acts on $X$ and $Y$ freely, then there is an inequality 
$$
\mathrm{ind}_{K}(X) \leq \mathrm{ind}_{G}(X).
$$ 
\end{enumerate}
}
\end{proposition}

\begin{proof}
(i)~Let $n:=\mathrm{coind}_{G}(X)$, and take a $G$-map $\phi \colon E_nG \to X$. Then the composition $\phi \circ E_ni \colon E_nK \to E_nG \to X$ is a $K$-map where the map $E_ni \colon E_nK \to E_nG$ is a $K$-map induced by the inclusion homomorphism $i \colon K \to G$. Thus we get the inequality $\mathrm{coind}_{K}(X) \geq n=\mathrm{coind}_{G}(X)$. Let $l:=\mathrm{wgt}_{G}(X)$. Note that since any subgroup of  $G=(\mathbb{Z}/p)^m$ is a direct summand, there is a retraction $r \colon G \to K$ which induces a commutative diagram
\begin{equation*}
\xymatrix@C=54pt{
	EG \times_G X \ar[r]^-{r \times _G \mathrm{id}}\ar[d] & EK \times_K X \ar[d]\\
	BG \ar[r]^-{Br} & BK. 
}
\end{equation*}
Then we get a commutative diagram
\begin{equation*}
\xymatrix@C=54pt{
	H^k(EK \times_K X; \mathbb{Z}/p)\ar[r]^-{(r \times_G \mathrm{id})^*} & H^k(EG \times_G X; \mathbb{Z}/p)\\
	H^k(BH; \mathbb{Z}/p) \ar[u] \ar[r]^-{(Br)^*} & H^k(BG; \mathbb{Z}/p). \ar[u] 
}
\end{equation*}
Note that the map $(Br)^* \colon H^k(BK; \mathbb{Z}/p) \to H^k(BG; \mathbb{Z}/p)$ is injective since the homomorphism $r \colon G \to K$ is a retraction. Since the map $H^k(BG; \mathbb{Z}/p) \to H^k(EG \times_G X; \mathbb{Z}/p)$ is injective for $k \leq l$, the map $H^k(BH; \mathbb{Z}/p) \to H^k(EK \times_K X; \mathbb{Z}/p)$ must be injective for $k \leq l$. Thus we get the inequality $\mathrm{wgt}_{K}(X) \geq l=\mathrm{wgt}_{G}(X)$.

(ii)~Let $m:=\mathrm{ind}_{G}(X)$, and take a $G$-map $\psi \colon X \to E_mG$. Then the composition $E_mr \circ \psi \colon X \to E_mG \to E_mK$ is a $K$-map. Thus we get the inequality $\mathrm{ind}_{K}(X) \leq m=\mathrm{ind}_{G}(X)$. 

\end{proof}

\begin{proposition}
\label{property}
\textit{
Let $X$ be a $G$-space. 
\begin{enumerate}[(i)]
  \setlength{\itemsep}{-0.75mm}
	\item There are equalities 
$$
\mathrm{coind}_{G}(E_nG)=\mathrm{ind}_{G}(E_nG)=n. 
$$
	\item If $X$ is $(n-1)$-connected, then there is an inequality
$$
\mathrm{coind}_{G}(X) \geq n. 
$$
	\item If $X$ is a free $G$-complex, then there is an inequality
$$
\mathrm{ind}_{G}(X) \leq \dim X. 
$$
\end{enumerate}
}
\end{proposition}

\begin{proof}
This proposition follows from the proof of \cite[Proposition 6.2.4]{matousek}. 
\end{proof}


\section{Proof of the main theorems}
In this section, we give proofs of the main theorems. We first consider the case that $K$ is a full simplex $\Delta^{m-1}$. 

\begin{lemma}
\label{Delta}
\textit{
Let $G<(\mathbb{Z}/2)^m$ be a non-trivial subgroup. Then there are equalities
$$
\mathrm{coind}_G(\mathbb{R}Z_{\Delta^{m-1}}) = \mathrm{wgt}_G(\mathbb{R}Z_{\Delta^{m-1}}) = \infty. 
$$
}
\end{lemma} 

\begin{proof}
Since $\mathbb{R}Z_{\Delta^{m-1}}=(D^1)^m$ is contractible, there are inequalities
$$
n \leq \mathrm{coind}_G(\mathbb{R}Z_{\Delta^{m-1}}) \leq \mathrm{wgt}_G(\mathbb{R}Z_{\Delta^{m-1}})
$$
for any $n\geq0$ by Proposition \ref{property1} and Proposition \ref{property}. 
\end{proof}

\begin{lemma}
\label{connectivity}
\textit{
Let $K$ be a simplicial complex on the vertex set $[m]$. If $K$ is $q$-neighborly, then $\mathbb{R}Z_K$ is $q$-connected. 
}
\end{lemma}

\begin{proof}
Since $K$ is $q$-neighborly, the $(q+1)$-skeleton of $\mathbb{R}Z_K$ equals to the $(q+1)$-skeleton of $\mathbb{R}Z_{\Delta^{m-1}}=(D^1)^m$ which is homotopy equivalent to a wedge of $(q+1)$-spheres. Thus by the cellular approximation theorem, we get 
$$
\pi_i \left( \mathbb{R}Z_K \right) = \pi_i \left( (D^1)^m \right) = 0 \quad \text{for $i \leq q$}, 
$$ 
completing the proof. 
\end{proof}

\begin{corollary}
\label{lower}
\textit{
Let $K$ be a simplicial complex on the vertex set $[m]$ and $G<(\mathbb{Z}/2)^m$ be a non-trivial subgroup. Then there is an inequality 
$$
\delta(K) \leq \mathrm{coind}_{G}(\mathbb{R}Z_K). 
$$
}
\end{corollary}

\begin{proof}
Let $q:=\delta(K)$. Since $K$ is $(q-1)$-neighborly, $\mathbb{R}Z_K$ is $(q-1)$-connected by Lemma \ref{connectivity}. Thus the proof is completed by Proposition \ref{property}.  
\end{proof}

\begin{lemma}
\label{diagonal}
\textit{
Let $K$ be a simplicial complex on the  vertex set $[m]$ Suppose that $K$ is not a full simplex, and $G=\mathbb{Z}/2$ acts on $\mathbb{R}Z_K$ diagonally. Then there are equalities
$$
\mathrm{coind}_G(\mathbb{R}Z_K)=\mathrm{wgt}_G(\mathbb{R}Z_K)=\mathrm{ind}_G(\mathbb{R}Z_K)=\delta(K). 
$$ 
}
\end{lemma}

\begin{proof}
We first note that the diagonal action of $G$ on $\mathbb{R}Z_K$ is free since $K$ is not a full simplex. It suffices to show that there are inequalities
$$
\delta(K) \leq \mathrm{coind}_G(\mathbb{R}Z_K) \leq \mathrm{wgt}_G(\mathbb{R}Z_K) \leq \mathrm{ind}_G(\mathbb{R}Z_K) \leq \delta(K). 
$$
The first inequality follows from Corollary \ref{lower}. The second and the third inequalities hold by Proposition \ref{property1}. If we set $q:=\delta(K)$, then there is a minimal non-face $I=\{i_1, \dots, i_{q+1}\} \subset [m]$ of $K$ which yields a $G$-map
$$
\mathbb{R}Z_K \to \mathbb{R}Z_{K_I}=\mathbb{R}Z_{\partial \Delta^q}=S^q \quad (x_1, \dots, x_m) \mapsto (x_{i_1}, \dots, x_{i_{q+1}}), 
$$
implying the last inequality by Propositions \ref{map} and \ref{property}. 
\end{proof}

\begin{proof}[proof of Theorems \ref{main1} and \ref{main2}]
Let us consider the inclusion $\mathbb{R}Z_{K_{\mathrm{supp}(G)}} \to \mathbb{R}Z_{K}$ and the projection $\mathbb{R}Z_K \to \mathbb{R}Z_{K_{\mathrm{supp}(G)}}$. Since these maps are $G$-maps, the $G$-index, $G$-coindex and $G$-weight of $\mathbb{R}Z_K$ are the same as those of $\mathbb{R}Z_{K_{\mathrm{supp}(G)}}$ respectively by Proposition \ref{map}. Thus we may assume that $K_{\mathrm{supp}(G)}=K$, or equivalently, $\mathrm{supp}(G)=[m]$.

First, we consider the case that $\mathrm{rank}\,G=1$. Then the action of $G$ on $\mathbb{R}Z_K$ is diagonal since $\mathrm{supp}(G)=[m]$. If the diagonal action of $G$ is not free, then $K$ must be a full simplex $\Delta^{m-1}$. By Lemma \ref{Delta}, there are equalities
$$
\mathrm{coind}_G(\mathbb{R}Z_K) = \mathrm{wgt}_G(\mathbb{R}Z_K) = \infty =\delta(K). 
$$
If the action of $G$ on $\mathbb{R}Z_K$ is free, then $K$ is not a full simplex. By Lemma \ref{diagonal}, there are equalities
$$
\mathrm{coind}_G(\mathbb{R}Z_K) = \mathrm{wgt}_G(\mathbb{R}Z_K) = \mathrm{ind}_G(\mathbb{R}Z_K) =\delta(K). 
$$
Thus the proof of the $\mathrm{rank}\,G=1$ case is completed. 

Next, let $G < (\mathbb{Z}/2)^m$ be a subgroup whose rank is not necessarily one. For each non-trivial element $g \in G$, there are inequalities 
$$
\delta(K) \leq \mathrm{coind}_{G}(\mathbb{R}Z_K) \leq \mathrm{wgt}_{G}(\mathbb{R}Z_K) \leq \mathrm{wgt}_{\langle g \rangle}(\mathbb{R}Z_K) =
\delta(K_{\mathrm{supp}(g)})
$$
by Proposition \ref{subgroup} and Corollary \ref{lower}. If the action of $G$ on $\mathbb{R}Z_K$ is free, then there are an inequalities 
$$
\delta(K_{\mathrm{supp}(g)}) = \mathrm{ind}_{\langle g \rangle}(\mathbb{R}Z_K) \leq \mathrm{ind}_{G}(\mathbb{R}Z_K) \leq \dim \mathbb{R}Z_K = \dim K+1. 
$$
by Proposition \ref{subgroup} and Lemma \ref{property}. The proof is completed. 
\end{proof}


\section{Upper bound of the $G$-index}

In this section, we prove Proposition \ref{retract} by constructing a $(\mathbb{Z}/2)^m$-subcomplex $X$ of $\mathbb{R}Z_K$, which is a $(\mathbb{Z}/2)^m$-equivariant deformation retract. If $\dim X \leq \dim K$, then we have an inequality 
$$
\text{ind}_G(\mathbb{R}Z_K) \leq \text{ind}_G(X) \leq \dim X \leq \dim K
$$ 
by Proposition \ref{property}, implying Proposition \ref{retract}.

Let $I=[0, 1]$ be a unit interval. Since a pair $(I^q, I^{q-1} \times \{0\} \cup \partial I^{q-1} \times I)$ is an NDR pair, we can choose a deformation retraction $\rho \colon I^q \to I^{q-1} \times \{0\} \cup \partial I^{q-1} \times I$ for each $q=1, \dots, m$. Let $\rho_j(x)$ denote the $j$-th entry of $\rho(x)$. For each $j=1, \dots, q$, we define a map $r_j \colon (D^1)^q \to D^1$ as
\begin{equation*}
r_j(x_1, \dots, x_q)=
	\begin{cases}
		\rho_j(|x_1|, \dots, |x_q|) & (x_j \geq 0) \\
		-\rho_j(|x_1|, \dots, |x_q|) & (x_j \leq 0), 
	\end{cases}
\end{equation*}
which is well-defined since $\rho_j(|x_1|, \dots, |x_q|)=|x_j|=0$ when $|x_j|=0$. Let us define a map $r \colon (D^1)^q \to (D^1)^q$ as
$$
r(x):=(r_1(x), \dots, r_q(x)), 
$$
which is $(\mathbb{Z}/2)^m$-equivariant by construction. 

\begin{lemma}
\label{r}
\textit{
The image of the map $r$ is a space
$$
Z:=\left( \bigcup_{j=1}^{q} (D^1)^{j-1} \times \{0\} \times (D^1)^{q-j} \right) \cup \left( \bigcup_{j=1}^{q-1} (D^1)^{j-1} \times S^0 \times (D^1)^{q-j} \right), 
$$
and a map $(D^1)^q \to Z ~ (x \mapsto r(x))$ is a deformation retraction. 
}
\end{lemma}

\begin{proof}
Let $x=(x_1, \dots, x_q)$ be any element of $(D^1)^q$. We first note that the image of the map $\rho$ is 
$$
I^{q-1} \times \{0\} \cup \partial I^{q-1} \times I = \left( \bigcup_{j=1}^q I^{j-1} \times \{0\} \times I^{q-j} \right) \cup \left( \bigcup_{j=1}^{q-1} I^{j-1} \times \{1\} \times I^{q-j} \right).  
$$
If $\rho(|x_1|, \dots, |x_q|) \in \bigcup_{j=1}^q I^{j-1} \times \{0\} \times I^{q-j}$, then $r_j(x)=0$ for some $j=1, \dots, q$ since $\rho_j(|x_1|, \dots, |x_q|)=0$ for some $j=1, \dots, q$, implying that 
$$
r(x) \in \bigcup_{j=1}^{q} (D^1)^{j-1} \times \{0\} \times (D^1)^{q-j}. 
$$
If $\rho(|x_1|, \dots, |x_q|) \in \bigcup_{j=1}^{q-1} I^{j-1} \times \{1\} \times I^{q-j}$, then $r_j(x)=1$ for some $j=1, \dots, q-1$ since $\rho_j(|x_1|, \dots, |x_q|)=1$ for some $j=1, \dots, q-1$, implying that 
$$
r(x) \in \bigcup_{j=1}^{q-1} (D^1)^{j-1} \times S^0 \times (D^1)^{q-j}. 
$$ 
Thus it is proved that $r(x) \in Z$. Conversely, if $x=(x_1, \dots, x_q) \in Z$, then $(|x_1|, \dots, |x_q|) \in I^{q-1} \times \{0\} \cup \partial I^{q-1} \times I$, implying that $\rho(|x_1|, \dots, |x_q|)=(|x_1|, \dots, |x_q|)$ since the map $\rho \colon I^q \to I^{q-1} \times \{0\} \cup \partial I^{q-1} \times I$ is a retraction. Thus for each $j=1, \dots, q$, 
\begin{equation*}
r_j(x)=
	\begin{cases}
		\rho_j(|x_1|, \dots, |x_q|)=|x_j|=x_j & (x_j \geq 0) \\
		-\rho_j(|x_1|, \dots, |x_q|)=-|x_j|=x_j & (x_j \leq 0), 
	\end{cases}
\end{equation*}
so that $r(x)=x$. Therefore the image of the map $r$ is $Z$, and the map $r$ is identical on $Z$. Let a map $H \colon I^q \times I \to I^q$ be a homotopy for the deformation retraction $\rho$, and denote the $j$-th entry of $H(x, t)$ by $H_j(x, t)$. Then we define a homotopy $\tilde{H} \colon (D^1)^q \times I \to (D^1)^q$ as 
$$
\tilde{H}(x, t):=(\tilde{H}_1(x, t), \dots, \tilde{H}_q(x, t)), 
$$
where the map $\tilde{H}_j \colon  (D^1)^q \times I \to D^1$ is defined as
\begin{equation*}
\tilde{H}_j(x, t)=
	\begin{cases}
		H_j(|x_1|, \dots, |x_q|, t) & (x_j \geq 0) \\
		-H_j(|x_1|, \dots, |x_q|, t) & (x_j \leq 0) 
	\end{cases}
\end{equation*}
for each $j=1, \dots, q$. By construction, the homotopy $\tilde{H}$ implies that the map $(D^1)^q \to Z ~ (x \mapsto r(x))$ is a deformation retraction. 
\end{proof}

Let us also denote the map $(D^1)^q \to Z$ in Lemma \ref{r} by $r$. 

\begin{lemma}
\label{elementary collapse}
\textit{
Let $K$ be a simplicial complex and $K \searrow_{(\sigma, \tau)} K^{\prime}$  be an elementary collapse. Then there is a $(\mathbb{Z}/2)^m$-equivariant deformation retraction
$$
\tilde{r} \colon \mathbb{R}Z_K \to Y \cup \mathbb{R}Z_{K^{\prime}}
$$
where $Y$ is a $(\mathbb{Z}/2)^m$-subcomplex of $\mathbb{R}Z_K$ with $\dim Y= \dim \sigma$. 
}
\end{lemma}

\begin{proof}
We may assume that $\sigma=\{1, \dots, q\}$ and $\tau=\{1, \dots, q-1\}$ and consider a map $r \times \mathrm{id} \colon (D^1)^q \times (S^0)^{m-q} \to Z \times (S^0)^{m-q}$. The target space $Z \times (S^0)^{m-q}$ equals to 
	\begin{eqnarray*}
&&\left( \bigcup_{j=1}^{q} (D^1)^{j-1} \times \{0\} \times (D^1)^{q-j} \times (S^0)^{m-q} \right) \cup \left( \bigcup_{j=1}^{q-1} (D^1)^{j-1} \times S^0 \times (D^1)^{q-j}  \times (S^0)^{m-q}\right) \\
&=&\left( \bigcup_{j=1}^{q} (D^1)^{j-1} \times \{0\} \times (D^1)^{q-j} \times (S^0)^{m-q} \right) \cup \bigcup_{{\sigma}^{\prime}<\sigma, ~{\sigma}^{\prime} \neq \tau} (D^1, S^0)^{{\sigma}^{\prime}}. 
	\end{eqnarray*}
Note that since $K^{\prime}=K \setminus \{\sigma, \tau\}$, the space $(D^1, S^0)^{{\sigma}^{\prime}}$ is a subspace of $\mathbb{R}Z_{K^{\prime}}$ for any face $\sigma^{\prime}<\sigma$ except for $\tau$. Thus if we set $Y:=\bigcup_{j=1}^{q} (D^1)^{j-1} \times \{0\} \times (D^1)^{q-j} \times (S^0)^{m-q}$, then $Y$ is a $(\mathbb{Z}/2)^m$-complex of $\mathbb{R}Z_K$ with $\dim Y=q-1=\dim \sigma$, and we get a well-defined map
$$
(D^1, S^0)^{\sigma} 
\xrightarrow{r \times \mathrm{id}} Y \cup \bigcup_{{\sigma}^{\prime}<\sigma, ~{\sigma}^{\prime} \neq \tau} (D^1, S^0)^{{\sigma}^{\prime}} 
\xrightarrow{\mathrm{inclusion}} Y \cup \mathbb{R}Z_{K^{\prime}}, 
$$
which is identical on $\mathbb{R}Z_{K^{\prime}} \cap (D^1, S^0)^{\sigma}=\bigcup_{{\sigma}^{\prime}<\sigma, ~{\sigma}^{\prime} \neq \tau} (D^1, S^0)^{{\sigma}^{\prime}} \subset Z$ since
the map $r$ is identical on $Z$. Therefore we get a map 
$$
\tilde{r} \colon \mathbb{R}Z_K=\mathbb{R}Z_{K^{\prime}} \cup (D^1, S^0)^{\sigma} \to Y \cup \mathbb{R}Z_{K^{\prime}}, 
$$
which is a $(\mathbb{Z}/2)^m$-equivariant deformation retraction by construction. Thus the lemma is proved. 
\end{proof}

\begin{proof}[proof of Proposition \ref{retract}]
By assumption, there is a sequence of elementary collapses $K=K_0 \searrow_{(\sigma_1, \tau_1)} K_1 \searrow_{(\sigma_2, \tau_2)} \cdots \searrow_{(\sigma_l, \tau_l)} K_l=L$. For each $i=1, \dots, l$, by Lemma \ref{elementary collapse}, there is a deformation retraction 
$$
\tilde{r}_i \colon \mathbb{R}Z_{K_{i-1}} \to Y_i \cup \mathbb{R}Z_{K_i}
$$
where $Y_i$ is a $(\mathbb{Z}/2)^m$-subcomplex of $\mathbb{R}Z_{K_{i-1}}$. Let us consider a map 
$$
\mathbb{R}Z_{K_{i-1}}
\xrightarrow{\tilde{r}_i} Y_i \cup \mathbb{R}Z_{K_i}
\xrightarrow{\mathrm{inclusion}} \bigcup_{p=1}^{i-1} Y_p \cup Y_i \cup \mathbb{R}Z_{K_i}=\bigcup_{p=1}^i Y_p \cup \mathbb{R}Z_{K_i}. 
$$
We show that this composition is identical on a space $\bigcup_{p=1}^{i-1} Y_p \cap \mathbb{R}Z_{K_{i-1}}$. It suffices to show that the space $\bigcup_{p=1}^{i-1} Y_p \cap \mathbb{R}Z_{K_{i-1}}$ is a subset of a space $Y_i \cup \mathbb{R}Z_{K_i}$ since the map $\tilde{r}_i$ is identical on $Y_i \cup \mathbb{R}Z_{K_i}$. Let $x$ be an arbitrary element of $\bigcup_{p=1}^{i-1} Y_p \cap \mathbb{R}Z_{K_{i-1}}$. Then $x \in Y_p$ for some $p = 1, \dots, i-1$, which means that the $j$-th entry of $x$ is zero for some $j \in \sigma_p$. Thus if $x \in \mathbb{R}Z_{K_{i-1}} \setminus \mathbb{R}Z_{K_i}$, then it is necessary that $x \in (D^1, S^0)^{\sigma_i}$, implying that $x \in Y_i$ since $j \in \sigma_p \cap \sigma_i$. Therefore the above argument shows that $x$ is in $Y_i \cup \mathbb{R}Z_{K_i}$. Thus we get a well-defined map 
$$
\bigcup_{p=1}^{i-1} Y_p \cup \mathbb{R}Z_{K_{i-1}} \to \bigcup_{p=1}^i Y_p \cup \mathbb{R}Z_{K_i}, 
$$
which is a $(\mathbb{Z}/2)^m$-equivariant deformation retraction by construction. Then the composition of maps
$$
\mathbb{R}Z_K \to Y_1 \cup \mathbb{R}Z_{K_1} \to \cdots \to \bigcup_{p=1}^l Y_p \cup \mathbb{R}Z_{K_l}
$$
is a $(\mathbb{Z}/2)^m$-equivariant deformation retraction. Thus if we set $X:=\bigcup_{p=1}^l Y_p \cup \mathbb{R}Z_{K_l}$, then there is an inequality
$$
\dim X = \max_{1 \leq p \leq l} \left\{ \dim \mathbb{R}Z_{K_l}, \dim Y_p \right\}  \leq \dim K
$$
since there are inequalities $ \dim K_l=\dim L<\dim K $ and $ \dim Y_p= \dim \sigma_p \leq \dim K $ for all $p=1, \dots, l$. Therefore we get inequalities
$$
	\text{ind}_G(\mathbb{R}Z_K) \leq \text{ind}_G(X) \leq \dim X \leq \dim K
$$
by Proposition \ref{map}, completing the proof.  
\end{proof}

\begin{example}
	Let $K$ be a cone of the boundary of an $(m-2)$-simplex $\Delta^{m-2}$ and $G=\mathbb{Z}/2$ be a group acting on $\mathbb{R}Z_K$ diagonally. Then $\dim K=m-1$ and $K$ can collapse to a point. However, Theorem \ref{main1} implies that 
$$
\text{ind}_G(\mathbb{R}Z_K)=\delta(K)=m-1=\dim K. 
$$
Thus this example shows that the upper bound in Proposition \ref{retract} is best possible. 
\end{example}


\section{Lower bound of the $G$-coindex and the flag number}

In this section, we discuss a lower bound for the $G$-coindex. The goal is to prove Corollaries \ref{cor1} and \ref{cor2}. 

\begin{lemma}
\label{supp}
\textit
{
Let $g, h$ be non-trivial elements of $G<(\mathbb{Z}/2)^m$, and take elements $i \in \mathrm{supp}(g)$ and $j \in \mathrm{supp}(h)$. If $j \notin \mathrm{supp}(g)$ and $i \notin \mathrm{supp}(h)$, then $i, j \in \mathrm{supp}(gh)$. 
}
\end{lemma}

\begin{proof}
By assumption, the $i$-th and $j$-th entry of $g$ are $-1$ and $1$ respectively, and  the $i$-th and $j$-th entry of $h$ are $1$ and $-1$ respectively. Thus both of the $i$-th and $j$-th entry of $gh$ are $-1$, so that $i, j \in \mathrm{supp}(gh)$. 
\end{proof}

\begin{proof}[proof of Corollary \ref{cor1}]
Suppose that $\delta(K_{\text{supp}(G)})=1$. Then there are two elements $i, j \in \text{supp}(G)$ such that $\{i, j\} \notin K_{\text{supp}(G)}$. If we show that $\{i, j\} \in \text{supp}(g_0)$ for some $g_0 \in G \setminus \{1\}$, then $\{i, j\} \notin K_{\text{supp}(g_0)}$ since $\{i, j\} \notin K_{\text{supp}(G)}$ and $K_{\text{supp}(g_0)} \subset K_{\text{supp}(G)}$, that is, $\delta(K_{\text{supp}(g_0)})=1$. By Theorem \ref{main2}, there are inequalities 
\begin{equation*}
	\begin{split}
1=\delta(K_{\mathrm{supp}(G)}) &\leq \mathrm{coind}_{G}(\mathbb{R}Z_K) \leq \mathrm{wgt}_{G}(\mathbb{R}Z_K) \\
&\leq \min_{g \in G \setminus \{1\}} \delta(K_{\mathrm{supp}(g)}) \leq \delta(K_{\mathrm{supp}(g_0)})=1, 
	\end{split}
\end{equation*}
so that we obtain equalities $\mathrm{coind}_{G}(\mathbb{R}Z_K) = \mathrm{wgt}_{G}(\mathbb{R}Z_K)=1$. Thus it suffices to show that there is some element $g_0 \in G$ such that $i, j \in \text{supp}(g_0)$. Since $\text{supp}(G)=\bigcup_{g \in G}\text{supp}(g)$, we can choose elements $g_i, g_j \in G \setminus \{1\}$ such that $i \in \text{supp}(g_i), j \in \text{supp}(g_j)$.  If $j \in \text{supp}(g_i)$ or $i \in \text{supp}(g_j)$, then $i, j \in \text{supp}(g_i)$ or $i, j \in \text{supp}(g_j)$ respectively. If $j \notin \text{supp}(g_i)$ and $i \notin \text{supp}(g_j)$, then $i, j \in \text{supp}(g_ig_j)$ by Lemma \ref{supp}. Thus the proof is completed. 
\end{proof}

To prove Corollary \ref{cor2},  let us define the combinatorial notion \textit{flag number} $\text{flag}(K)$ of a simplicial complex $K$ as the maximum number of the dimension of minimal non-faces of $K$, that is, 
$$
\mathrm{flag}(K):=\max \{|I|-1 \mid \text{$I$ is a minimal non-face of $K$}\}
$$
If $K$ is a full simplex, then we set $\mathrm{flag}(K)=\infty$. It is easy to see that a simplicial complex $K$ is flag if and only if $K$ is a full simplex or its flag number equals to 1.

\begin{lemma}
\label{full subcomplex}
\textit
{
Let $K$ be a simplicial complex on the vertex set $[m]$ and $I \subset [m]$ be a non-face of $K$. Then there are inequalities 
$$
\delta(K) \leq \delta(K_I) \quad \text{and} \quad \mathrm{flag}(K_I) \leq \mathrm{flag}(K). 
$$
}
\end{lemma}

\begin{proof}
We first note that $K_I$ is not a full simplex by assumption. Let $q:=\delta(K_I)$. Then there is a non-face $J$ of $K_I$ with $|J|=q+1$. By definition of full subcomplex, $J$ is also a non-face of $K$. Thus there is an inequality 
$$
\delta(K) \leq q = \delta(K_I).
$$
Let $q^{\prime}:=\mathrm{flag}(K)$ and take an arbitrary minimal non-face $J^{\prime}$ of $K_I$. Then $J^{\prime}$ is also a minimal non-face of $K$, so that $|J^{\prime}|-1 \leq q^{\prime}$ by definition of flag number. Thus there is an inequality 
$$
\mathrm{flag}(K_I) \leq q^{\prime} = \mathrm{flag}(K). 
$$
The proof is completed. 
\end{proof}

\begin{corollary}
\label{flag}
\textit
{
Let $K$ be a simplicial complex. Then there is an inequality 
$$
\delta(K) \leq \mathrm{flag}(K). 
$$
}
\end{corollary}

\begin{proof}
If $K$ is a full simplex, then this corollary is trivial since both $\delta(K)$ and $\mathrm{flag}(K)$ are $\infty$. If $K$ is not a full simplex, then there is a minimal non-face $I$ of $K$. Then $K_I$ is the boundary of a full simplex. By definition, the $\delta$-number and the flag number of the boundary of a full simplex have the same number, so that there are inequalities 
$$
\delta(K) \leq \delta(K_I) = \mathrm{flag}(K_I) \leq \mathrm{flag}(K) 
$$ 
by Lemma \ref{full subcomplex}, completing the proof. 
\end{proof}

\begin{proposition}
\label{same order}
\textit
{
Let $K$ be a simplicial complex. Then the following are equivalent. 
\begin{enumerate}[(i)]
	\item $\delta(K)=\mathrm{flag}(K)$. 
	\item Each minimal non-face of $K$ has the same order.  
\end{enumerate}
}
\end{proposition}

\begin{proof}
If $K$ is a full simplex, then this proposition is trivial, so we may assume that $K$ is not a full simplex. Suppose that $\delta(K)=\mathrm{flag}(K)=q$ and take an arbitrary minimal non-face $I$ of $K$. Since $\delta(K)=q$, every non-face of $K$ has an order larger than or equals to $q+1$, so that there is an inequality 
$$
|I| \geq q+1. 
$$
Since $\mathrm{flag}(K)=q$, every minimal non-face of $K$ has an order smaller than or equals to $q+1$, so that there is also an equality 
$$
|I| \leq q+1. 
$$
Therefore each minimal non-face of $K$ has the same order $q+1$. Conversely, suppose that each minimal non-face of $K$ has the same order $q+1$. Since there is no non-face of $K$ with an order smaller than $q+1$, there is an inequality
$$
\delta(K) \geq q. 
$$
On the other hand, since every minimal non-face $I$ of $K$ satisfies an inequality $|I|-1 \leq q$, there is also an inequality 
$$
\mathrm{flag}(K) \leq q. 
$$
By Corollary \ref{flag}, we get inequalities 
$$
q \leq \delta(K) \leq \mathrm{flag}(K) \leq q. 
$$
Thus the equality $\delta(K) = \mathrm{flag}(K)$ holds, completing the proof. 
\end{proof}

\begin{proof}[proof of Corollary \ref{cor2}]
Since there is some element $g_0 \in G \setminus \{1\}$ such that $\text{supp}(g_0) \notin K$, $K_{\text{supp}(G)}$ cannot be a full simplex. By assumption, there is an equality 
$$
\delta(K_{\text{supp}(G)}) = \mathrm{flag}(K_{\text{supp}(G)})
$$ 
by Proposition \ref{same order}. Since $K_{\text{supp}(g_0)}$ is not a full simplex, there are inequalities 
$$
\delta(K_{\text{supp}(G)}) \leq \delta(K_{\text{supp}(g_0)}) \leq \mathrm{flag}(K_{\text{supp}(g_0)}) \leq \mathrm{flag}(K_{\text{supp}(G)}) 
$$
by Proposition \ref{flag} and Lemma \ref{full subcomplex}. Thus we get an equality 
$$
\delta(K_{\text{supp}(G)}) = \delta(K_{\text{supp}(g_0)}). 
$$
Finally, let us consider inequalities
$$
\delta(K_{\mathrm{supp}(G)}) \leq \mathrm{coind}_{G}(\mathbb{R}Z_K) \leq \mathrm{wgt}_{G}(\mathbb{R}Z_K) \leq \min_{g \in G \setminus \{1\}} \delta(K_{\mathrm{supp}(g)}) \leq \delta(K_{\mathrm{supp}(g_0)})
$$
in Theorem \ref{main2}. Since the left and right sides are equal, we get the equalities 
$$
\mathrm{coind}_{G}(\mathbb{R}Z_K) = \mathrm{wgt}_{G} = \delta(K_{\mathrm{supp}(G)}), 
$$ 
completing the proof. 
\end{proof}


\section*{Acknowledgement}
The author is deeply grateful to Daisuke Kishimoto for much helpful and valuable advice.


 \end{document}